\documentclass[11pt,bezier]{article}
\usepackage{tikz}
\usepackage{float}
\usepackage{amsmath,amssymb,amsfonts,euscript,graphicx}

\newtheorem{preproof}{{\bf \indent Proof.}}

\newenvironment{proof}[1]{\begin{preproof}{\rm
               #1}\hfill{$\Box$}}{\end{preproof}}

\newtheorem{cor}{\bf\indent Corollary}[section]
\newtheorem{example}{\bf\indent Example}[section]
\newtheorem{thm}{{\bf\indent Theorem}}[section]
\newtheorem{prop}{\bf\indent Proposition}[section]
\newtheorem{remark}{\bf\indent Remark}[section]
\newtheorem{lem}{\bf\indent Lemma}[section]

\title{\bf \large  Strong resolving graph of \\the intersection graph in commutative rings}

\author{{\normalsize  {\sc E. Dodongeh${}^{\mathsf{a}}$,  A. Moussavi${}^{\mathsf{a}}$,     R. Nikandish${}^{\mathsf{b}}$\thanks{Corresponding author}}
}\vspace{3mm}\\
{\footnotesize{${}^{\mathsf{a}}$\it Department of Mathematics, University of Tarbiat Modarres,
Tehran, Iran}}\\
{\footnotesize{${}^{\mathsf{b}}$\it  Department of Mathematics,
Jundi-Shapur University of Technology,  Dezful,
Iran}}\\
{\footnotesize{${}^{\mathsf{}}$\it}}\\
{\footnotesize{$\mathsf{e.dodongeh@modares.ac.ir}$\quad\quad $\mathsf{moussavi.a@modares.ac.ir}$ \quad\quad
$\mathsf{r.nikandish@ipm.ir}$\quad\quad}}}

\begin{document}

\maketitle

\begin{abstract}
{\small The intersection graph of ideals associated with a commutative unitary ring $R$ is the graph $G(R)$ whose vertices all non-trivial ideals of $R$ and there exists an edge between distinct vertices if and only if the intersection of them is non-zero. In this paper, the structure of the resolving graph of $G(R)$ is characterized and as an application, we evaluate the strong metric dimension of $G(R)$.}
\end{abstract}
\begin{center}\section{Introduction}\end{center}
Metric and strong metric dimension of a graph are two of the most applicable parameters with several usages in robotic, computer science, chemistry, optimization  etc. Although these invariants  have been computed for some classes of well-known graphs, they are still the subjects  of many researches; for instance see \cite{abr, bai, ghala, Ji, Ve}. Among the reasons for considerable interest  in characterizing these parameters for graphs  associated with algebraic structures one may cite variety of uses and complexity of computations. Some
examples in this direction may be found in \cite{bakht, ali, dol, dol2,  ma, nili, nili2, Pirzada1, Pirzada2, zai}.  This paper has a such goal and aims to discuss the strong metric dimension in intersection graphs of ideals of commutative rings.
\par
 For graph theory terminology, we follow \cite{west}. Let $G=(V,E)$ be a graph with $V=V(G)$ as the vertex set and $E=E(G)$ as the edge set. A complete graph of order $n$  is denoted by $K_n$. Also, distance between two distinct vertices $x$ and $y$ is denoted by $d(x,y)$. By diam$(G)$, we mean the diameter of $G$.  Moreover, the induced subgraph by $V_0\subseteq V$ is  denoted by $G[V_0]$. The open and closed neighborhood of the vertex $x$ are denoted by  $N(x)$ and $N[x]$, respectively. The complement of $G$ is denoted by $\overline{G}$. The independence number and vertex cover number of the graph $G$ are denoted by $\beta(G)$ and  $\alpha(G)$, respectively.
 Let $S=\{v_1,v_2,\dots,v_k\}$ be an ordered subset of $V$ and $v\in V\setminus S$. Then the representation vector of  $v$ with respect to $S$ is denoted by $D(v|S)$ which is defined as follows: $D(v|S)=(d(v,v_1),d(v,v_2),\dots, d(v,v_k))$.  An ordered subset $S\subseteq V(G)$ is called \textit{resolving} provided that distinct vertices out of $S$ have different representation vectors with respect to $S$.  The cardinality of any
resolving set of minimum cardinality is called the \textit{metric dimension of} $G$ and denoted   by $dim_M(G)$.
 Two different vertices $u,v$ \textit{are mutually maximally distant} if $d(v, w) \leq d(u, v)$, for every $w \in N(u)$ and $ d(u, w) \leq d(u, v)$, for every $w \in N(v)$. For a graph $G$, \textit{the strong resolving graph of} $G$, is denoted by $G_{SR}$ and its vertex and edge set are defined as follow:  
$V(G_{SR})= \lbrace u \in V (G)|\,there~exists ~v \in V (G)  ~such~that ~u,  v ~are ~mutually ~maximally ~distant \rbrace$ and  $uv \in E(G_{SR})$  if  and  only  if $u$  and $v$ are  mutually maximally distant. Two vertices $u$ and $v$ are \textit{strongly resolved} by some vertex $w$ if either 
$d(w, u)$ is  equal to $d(w, v) + d(v, u)$  or $d(w, v)$ is equal to $d(w, u) + d(v, u)$. A set $W$ of vertices
is a \textit{strong resolving set of} $G$ if every two distinct vertices of $G$ are strongly
resolved by some vertex of $W$ and a minimum strong resolving set is called
 \textit{strong metric basis} and its cardinality is \textit{the strong metric dimension of} 
$G$. We denote the strong metric dimension of $G$, by $sdim(G)$.

\par
Throughout this paper, all rings are assumed to be commutative with identity. The set of all non-trivial ideals of $R$ is denoted by  $I(R)$. The ring $R$ is called  \textit{reduced} if it has no nilpotent elements other than $0_R$. For undefined notions in ring theory, we refer the reader to \cite{ati}.
\par
\textit{The intersection graph of ideals of a  ring} $R$, denoted by $G(R)$, is a simple and undirected graph whose vertex set is $I(R)$ and two distinct vertices are adjacent if and only if they have non-zero intersection. This graph was first introduced and studied by Chakrabarty et.al in  \cite{Chak} and many beautiful  properties  of it were obtained. Later, many researchers investigated different aspects of this concept; see for instance \cite{akb, mah, xu}. In \cite{nik}, the metric dimension of intersection graphs of rings was discussed.
In this paper, we characterize the structure of the resolving graph of $G(R)$ and as an application $sdim(G(R))$ is computed.
\section{$G(R)_{SR}$ and $sdim(G(R))$; $R$ is reduced}
 \noindent
 In this section, for a given ring $R$, first it is shown that $sdim_M(G(R))$ is finite if and only if $|I(R)|< \infty$. Then the graph $G(R)_{SR}$ and its vertex cover number are determined, when $R$ is reduced. Finally, $sdim(G(R))$ is given in an explicit formula.

\begin{prop}\label{dimfinite}
Let $R$ be a ring that is not a field. Then $sdim_{M}(G(R))< \infty$ if and only if $|I(R)|< \infty$.
\end{prop}
\begin{proof}
{First assume that $sdim_M(G(R))$ is finite. Then $dim_M(G(R))$ is finite too, as $dim_M(G(R))\leq sdim_M(G(R))$. Let $W=\{W_1,\ldots,W_n\}$ be a  metric basis for $G(R)$, where $n$ is a non-negative integer. By \cite[Theorem  2.1]{akb}, there exist $2^n$ possibilities for $D(X|W)$, for every $X \in V(G(R))\setminus W$. Thus $|V(G(R)| \leq 2^{n}+n$ and hence $R$ has finitely many ideals. The converse is trivial.}
\end{proof}

To compute  $sdim_M(G(R))$,  it is enough to consider rings with finitely many ideals, by Proposition \ref{dimfinite}. Therefore, from now on, we suppose that all rings $R$ have finitely many ideals. These rings are direct product of finitely many fields, if they are reduced.

We state a series of lemmas to calculate $sdim(G(R))$.

\begin{lem}\label{Oellermann} {\rm (\cite[Theorem 2.1]{oller})} For any connected graph $G$, $sdim_M(G)=\alpha(G_{SR})$.
\end{lem}

\begin{lem}\label{Gallai}
{\rm (Gallai$^{^,}$s theorem)} For any graph $G$ of order $n$, $\alpha(G)+\beta(G)=n$.
\end{lem}

The following remark introduces a notion which will be used several times.
\begin{remark}
	\label{dimfin}
	{\rm Let $R \cong \prod_{i=1}^{n}R_{i}$, where $R_{i}$ is a ring for every $1\leq i \leq n$, and $I=I_{1}\times \cdots \times I_{n} \in V(G(R))$. Then by $I^{c}=I_{1}^{c}\times \cdots \times I_{n}^{c}$, we mean a vertex of $G(R)$ such that $I_{i}^{c}= 0$ if $I_{i}\neq 0$ and $I_{i}^{c} =R_i$ if  $I_{i}= 0$, for every $1\leq i \leq n$. The ideal $I^{c}$ is called the complement of $I$. We note that different ideals may have a same complement.}
\end{remark}

\begin{lem}\label{lemma2g}
Let  $n\geq 2$ be a positive integer and $R\cong \prod_{i=1}^{n}\mathbb{F}_{i}$, where $\mathbb{F}_{i}$ is a field for every $1\leq i \leq n$. Then the following statements hold.\\
$1)$ $V(G(R)_{SR})=V(G(R))$.\\
$2)$ Suppose that $I, J \in V(G(R)_{SR})$, then  $IJ \in E(G(R)_{SR})$ if and only if $IJ \notin E(G(R))$.
\end{lem}
\begin{proof}
{
$1)$ For every $I=I_{1}\times\cdots \times I_{n}\in V(G(R))$, since $I\cap I^{c}= \emptyset$, we deduce that $d(I,I^{c})=2=diam(G(R))$. Thus $I$ and $I^{c}$ are mutually maximally distant and so $I\in V(G(R)_{SR})$ i.e., $V(G(R)_{SR})=V(G(R))$.\\
$2)$ First suppose that  $IJ \notin E(G(R))$. Since $d(I,J)=2$, obviously $IJ\in E(G(R)_{SR})$.\\
Conversely, suppose that $IJ \in  E(G(R)_{SR}$, for some $I, J \in V(G(R)_{SR})$. If $I\sim J$, then since $I\neq J$, we have $I\sim J^{c}$ or $J\sim I^{c}$. Thus  $d_{G(R)}(J,J^{c})=2 > 1=d(I,J)$ or $d_{G(R)}(I,I^{c})=2 > 1= d(I,J)$, and so $I, J$ are not mutually maximally distant, a contradiction. This completes the proof.
}
\end{proof}

Now, we have the following immediate corollary.
\begin{cor}\label{cor1}
Let  $n\geq 2$ be a positive integer and $R\cong \prod_{i=1}^{n}\mathbb{F}_{i}$, where $\mathbb{F}_{i}$ is a field for every $1\leq i \leq n$. Then $G(R)_{SR}=\overline{G(R)}$.
\end{cor}
The next example explains Corollary \ref{cor1} in case $n=3$.

\begin{example}
{{\rm
Suppose that $R\cong \prod_{i=1}^{3}\mathbb{F}_{i}$,  where $\mathbb{F}_{i}$ is a field for every $1\leq i \leq 3$. Thus $|V(G(R))|=6$. Let}\\
$$ V_{1}=\mathbb{F}_{1}\times \mathbb{F}_{2}\times 0, \,\,\,\,\  V_{2}=\mathbb{F}_{1}\times 0 \times \mathbb{F}_{3}, \,\,\,\,\   V_{3}= 0 \times \mathbb{F}_{2} \times \mathbb{F}_{3},  $$\\ 
$$   V_{4}=0 \times 0 \times \mathbb{F}_{3},\,\,\,\,\ V_{5}=0\times \mathbb{F}_{2}\times 0 , \,\,\,\,\ V_{6}=\mathbb{F}_{1}\times 0 \times 0$$.
\unitlength=1.5mm
\begin{center}
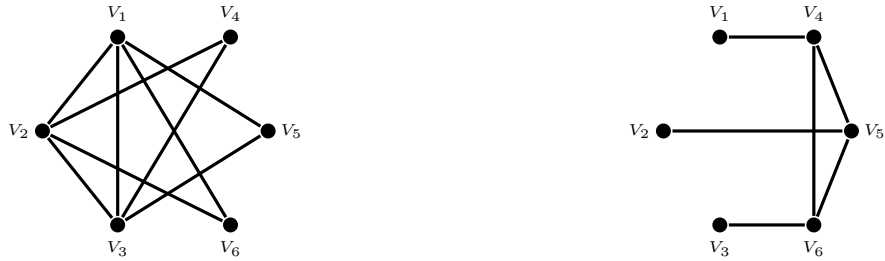
\begin{figure}[H]
 	
\begin{minipage}{.2\textwidth}
\begin{tikzpicture}
  [scale=0.25,every node/.style={circle,fill=black,inner sep=0pt},very thick]
  \node [label=above:{\tiny $V_{1}$},text width=2mm] (n1) at (0,5) {};
  \node [label=left:{\tiny $V_{2}$},text width=2mm] (n2) at (-4,0)  {};
  \node [label=below:{\tiny $V_{3}$},text width=2mm] (n3) at (0,-5)  {};
  \node [label=above:{\tiny $V_{4}$},text width=2mm]  (n4) at (6,5)  {};
  \node [label=right:{\tiny $V_{5}$},text width=2mm] (n5) at (8,0)  {};
  \node [label=below:{\tiny $V_{6}$},text width=2mm] (n6) at (6,-5)  {};
  \foreach \from/\to in {n1/n2,n1/n3,n1/n5,n1/n6,n2/n3,n2/n4,n2/n6,n3/n4,n3/n5}
    \draw (\from) -- (\to);
\end{tikzpicture}
\end{minipage}
\hspace{5cm}
\begin{minipage}{.25\textwidth}
\begin{tikzpicture}
  [scale=0.25,every node/.style={circle,fill=black,inner sep=0pt},very thick]
   \node [label=below:{\tiny $V_{6}$},text width=2mm] (n6) at (40,-5)  {};
   \node [label=left:{\tiny $V_{2}$},text width=2mm] (n2) at (32,0)  {};
   \node [label=below:{\tiny $V_{3}$},text width=2mm] (n3) at (35,-5)  {};
   \node [label=right:{\tiny $V_{5}$},text width=2mm] (n5) at (42,0)  {};
  \node [label=above:{\tiny $V_{1}$},text width=2mm] (n1) at (35,5)  {};
  \node [label=above:{\tiny $V_{4}$},text width=2mm] (n4) at (40,5)  {};
  \foreach \from/\to in {n1/n4,n2/n5,n3/n6,n4/n5,n4/n6,n5/n6}
    \draw (\from) -- (\to);
\end{tikzpicture}
\end{minipage}

 \caption{ $G(R)$ and $G(R)_{SR}$} \label{figure:fr1}
\end{figure}
\end{center}
}
\end{example}
Then  $\overline{G(R)}$  and  $G(R)_{SR}$ are identical.
\begin{lem}\label{dimprod2}
Let  $n\geq 2$ be a positive integer and $R\cong \prod_{i=1}^{n}\mathbb{F}_{i}$, where $\mathbb{F}_{i}$ is a field for every $1\leq i \leq n$. Then $\beta(G(R)_{SR})= 2^{n-1}-1$.
\end{lem}
\begin{proof}
{By Lemma \ref{lemma2g}, $V(G(R)_{SR})=V(G(R))$. Let $I=I_{1}\times\cdots \times I_{n}\in V(G(R)_{SR})$ and $NZC(I)$ be the number of zero components in $I$. Obviously, $1 \leq NZC(I) \leq n-1$. Assume that \\
$A_{1}=\lbrace I \in V(G(R)_{SR})| NZC(I)=1\rbrace$,\\
$A_{2}=\lbrace I \in V(G(R)_{SR})| NZC(I)=2\rbrace$,\\
\vdots
\\
and $A_{n-1}=\lbrace I \in V(G(R)_{SR})| NZC(I)=n-1\rbrace$.\\
It is easily seen that $V(G(R))=\cup_{i=1}^{n-1}A_{i}$ and $A_{i}\cap A_{j}=\emptyset$, for every $i\neq j$ and so $\lbrace A_{1}, \ldots, A_{n-1}\rbrace$ is a partition of $V(G(R))$. Take the following facts into observation:\\
{\bf{Fact 1.}} Let $I,J\in A_i$, for some $1 \leq i\leq n-1$. If $I$ is not  adjacent to $J$  in $G(R)_{SR}$,  then by Lemma \ref{lemma2g}, $I\sim J$ in $G(R)$.\\
{\bf{Fact 2.}} Let $1 \leq i\leq [\dfrac{n}{2}]-1$, for even $n$ and $1 \leq i \leq [\dfrac{n}{2}]$, otherwise. Then $S_i\subseteq A_i$ is the largest subset of $A_i$ such that $IJ\notin E(G(R)_{SR}) $, for  every $I,J\in S_i$ (indeed, $S_i$ is the largest independent subset of $A_i$ in $G(R)_{SR}[A_i]$). For every $I,J\in A_{i}$, we have $I\cap J \neq 0$, so by Fact 1, $I$ is not  adjacent to $J$ in $G(R)_{SR}$. Thus $|S_i|=|A_{i}|= {n \choose i}$.\\
{\bf{Fact 3.}} Let $ t=\dfrac{n}{2}$, where $n$ is even. Then for every $I \in A_t$, $I$ is only adjacent to $I^{c}$. Thus $|S_t|=\frac{|A_{t}|}{2}= \dfrac{{n \choose t}}{2}$, where $S_t\subseteq A_t$ is the largest subset of $A_t$ such that $IJ\notin E(G(R)_{SR}) $, for  every $I,J\in S_t$.\\
Now let $S^{\prime}=\cup _{i=1}^{[t]}S_{i}$ and $[t] \leq i\leq n-1$. Then  there exists $J\in S^{\prime}$ such that $I\cap J = 0$, for every $I\in A_{i}$. Thus $S^{\prime} \cap (\cup_{i=t+1}^{n-1}A_i)=\emptyset$. Furthermore, $|S^{\prime}|={n \choose 1} +\cdots+ {n \choose t}= 2^{n-1}-1$, where $n$ is odd and $|S^{\prime}|={n \choose 1} +\cdots+ {n \choose t-1}+ \dfrac{{n \choose t}}{2}= 2^{n-1}-1$, where $n$ is even. Hence $S^{\prime}$ is the largest independent subset of $V(G(R)_{SR}$ in $G(R)_{SR}$ and so $\beta(G(R)_{SR})=|S^{\prime}|= 2^{n-1}-1$. 
}
\end{proof}
\begin{thm}\label{dimprod22}
Let  $n\geq 2$ be a positive integer and $R\cong \prod_{i=1}^{n}\mathbb{F}_{i}$, where $\mathbb{F}_{i}$ is a field for every $1\leq i \leq n$. Then
$sdim(G(R)_{SR})=2^{n}- 2^{n-1}-1$.
\end{thm}
\begin{proof}
{The result follows from Lemmas \ref{Oellermann}, \ref{dimprod2},  Gallai's theorem and the fact that 
$|V(G(R)_{SR})|=2^{n}-2$.}
\end{proof}

\section{$G(R)_{SR}$ and $sdim(G(R))$; $R$ is non-reduced}
 

As it has been mentioned in Section $2$, we consider rings $R$ with finitely many ideals. Then there exists positive integer $m$ such that $R\cong R_1\times\cdots \times R_m$, where $(R_{i},m_{i})$ is a local Artinian ring, for all $1\leq i \leq m$. If every $m_{i}$ is principal, then by \cite[Propostion 8.8]{ati}, every $R_{i}$ is a principal ideal ring (PIR, for short) with finitely many ideals ({\bf we suppose throughout this section that $|I(R_i)|=n_i$, for $1\leq i\leq m$}). Moreover, ideals of every $R_{i}$ make an inclusion chain. 

In this section, we study the structure of $G(R)_{SR}$ and compute $sdim(G(R))$ for such rings $R$.

First, the case in which no fields appear in decomposition of $R$ is investigated.

\begin{remark}\label{dimfin1}
{\rm Suppose that  $R\cong \prod_{i=1}^{m}{R}_{i}$, where $R_{i}$ is a PIR  non-field for every $1\leq i\leq m$ and $m\geq 2$ is a positive integer. Assume that $I=I_{1}\times \cdots \times I_{m}$ and $J=J_{1}\times \cdots \times J_{m}$ are vertices of $G(R)$, where  $I_{i}, J_{i} \in R_{i}$, for every $1\leq i\leq m$. Define the relation  $\thicksim$ on  $V(G(R))$ as follows: $I\thicksim J$, whenever ``$I_{i}=0$ if and only if $ J_{i}=0$'', for every $1\leq i\leq m$.  It is easily seen that $\thicksim$ is an equivalence relation on $V(G(R))$. By $[I]$, we mean the equivalence class of $I$.
 Let $X_{1}$ and $X_{2}$ be two elements of $[X]$. Since $X_{1}\thicksim X_{2}$,  $X_{1}\cap X_{2}\neq 0$, i.e.,  $X_{1}$ and $X_{2}$ are adjacent. Moreover, $N(X_{1})=N(X_{2})$ and the number of these equivalence classes is $2^m-1$.}
\end{remark}
\begin{lem}\label{dimprod4}
Suppose that  $R\cong \prod_{i=1}^{m}{R}_{i}$, where $R_{i}$ is a PIR  non-field, for every $1\leq i\leq m$ and $m\geq 2$ is a positive integer. Then the following statements hold:\\
$1)$ $V(G(R)_{SR})=V(G(R))$.\\
$2)$ For every $I, J \in V(G(R)_{SR})$, if $[I]=[J]$, then  $IJ \in E(G(R)_{SR})$.\\
$3)$ For every $I, J\in V(G(R)_{SR})$, if $[I]\neq [J]$, then $IJ \in E(G(R)_{SR})$ if and only if $IJ \notin E(G(R))$.
\end{lem}
\begin{proof}{
$1)$  It is enough to show that $V(G(R))\subseteq V(G(R)_{SR})$. Let $I=I_{1}\times\cdots \times I_{m}\in V(G(R))$, $NZC(I)$ be the number of zero components of $I$ and $A_{i}=\lbrace I \in V(G(R)| NZC(I)=i\rbrace$, for $0\leq i \leq m-1$. Then $V(G(R))=\cup _{i=0}^{m-1}A_{i}$.
Suppose that $I=I_{1}\times\cdots \times I_{m}\in V(G(R))\setminus A_{0}$. Since $d(I,I^{c})=2=diam(G(R))$, we conclude that $I, I^{c}$ are mutually maximally distant and so $I\in V(G(R)_{SR})$. Now, let $I \in A_{0}$. Then $d(I, V)= d(J,V)=1$, for every  $J \in A_{0}$ and $V\in V(G(R))\setminus \lbrace I, J \rbrace$. Thus $I, J$ are mutually maximally distant and so $I\in V(G(R)_{SR})$.  \\
$2)$ If $[I]=[J] \subset V(G(R)_{SR})$, then by Remark \ref{dimfin1}, $N(I)=N(J)$. Thus $I, J$ are mutually maximally distant and so $IJ \in E(G(R)_{SR})$. \\
$3)$ If $IJ \notin E(G(R))$, then clearly $IJ \in E(G(R)_{SR})$. To prove the other side, suppose to the contrary, $IJ \in E(G(R))$. Since $[I]\neq [J]$, if $[I]=A_{0}$ or $[J]=A_{0}$, then $d(I,I^{c})=2> d(I,J)=1$ or $d(J,J^{c})=2> d(I,J)=1$. Thus $I, J$ are not mutually maximally distant and so $IJ \notin E(G(R)_{SR})$, else since $[I]\neq [J]$, we conclude that $I\sim J^{c}$ or $J \sim I^{c}$. Thus $d(I,I^{c})=2> d(I,J)=1$ or $d(J,J^{c})=2> d(I,J)=1$. Hence $I, J$ are not mutually maximally distant and $IJ\notin E(G(R)_{SR})$, a contradiction.   
}
\end{proof}
\begin{lem}\label{dimprod5}
Suppose that  $R\cong \prod_{i=1}^{m}{R}_{i}$, where $R_{i}$ is a PIR non-field for every $1\leq i\leq m$ and $m\geq 2$ is a positive integer. Then $G(R)_{SR}= K_{\Pi_{i=1}^{m}(n_{i}+1)-1}+ H$, where $H$ is a connected graph. 
\end{lem}
\begin{proof}{
 Using the notations in the proof of Lemma \ref{dimprod4}, $V(G(R)_{SR})=V(G(R))$. If $I, J\in A_{0}$, then $[I]=[J]$ and so $IJ\in E(G(R)_{SR})$. Thus the induced subgraph $G(R)[A_{0}]$ is complete. Also, by Lemma \ref{dimprod4}, if $I\in A_{0}$ and  $J\in V(G(R)_{SR})\setminus A_{0}$, then $IJ\notin E(G(R)_{SR})$. Furthermore, $|A_{0}|= \Pi_{i=1}^{m}(n_{i}+1)-1$. Thus $G(R)[A_{0}]= K_{\Pi_{i=1}^{m}(n_{i}+1)-1}$.\\
Next, we show that $H$ is a connected graph, where $V(H)=\cup _{i=1}^{m-1}A_{i}$.  
  We have to find a path between arbitrary vertices $I=I_1\times \cdots \times I_m$ and $J=J_1\times \cdots \times J_m$ in $V(H)$. To see this, we consider the following cases:\\
 \textbf{Case 1.} $[I]= [J]$.\\
 If $[I]= [J]$, then by Lemma \ref{dimprod4}, $I$ and $J$ are adjacent in $G(R)_{SR}$.\\
 \textbf{Case 2.} $[I]\neq[J]$.\\
 If $IJ\notin E(G(R))$, then by Lemma \ref{dimprod4}, $IJ\in E(G(R)_{SR})$. Thus suppose that $IJ\in E(G(R))$, so $I \cap J\neq 0$. If $I \subset J$ or $J \subset I$, then there exists $1 \leq i \leq n$ such that $I_{i}=J_{i}=0$, as $I,J \notin A_0$. In this case $I \sim V \sim J$, where $V= 0\times \cdots \times 0 \times R_{i}\times 0 \times \cdots \times 0$. Thus we may assume that $I \nsubseteq J$ and $J \nsubseteq I$. Hence there exist $1\leq i\neq j \leq m$ such that $I_{i}\neq 0 \neq J_{j}$ and $I_{j}= 0 = J_{i}$. In this case $I \sim V_{1}\sim V_{2} \sim J$, where $V_{1}= 0\times \cdots \times 0 \times R_{j}\times 0 \times \cdots \times 0$ and $V_{2}= 0\times \cdots \times 0 \times R_{i}\times 0 \times \cdots \times 0$.
 Thus $H$ is a connected graph. 
}
\end{proof}
The next example explains Lemma \ref{dimprod5} in case $n=2$.

\begin{example}
{{\rm
Suppose that $R\cong R_{1}\times R_{2}$,  where $R_{i}$ is a PIR non-field for $ i=1, 2$. Let $I(R_{i})=\lbrace I_{i1}, I_{i2}\rbrace$, for $i=1,2$. Thus $|V(G(R))|=14$. Suppose that}\\
$$ V_{1}=R_{1} \times 0, \,\,\,\,\  V_{2}= 0 \times R_{2}, \,\,\,\,\   V_{3}= I_{11} \times 0, \,\,\,\,\ V_{4}=I_{11} \times I_{21},$$\\
$$V_{5}=I_{11}\times I_{22}, \,\,\,\,\ V_{6}=I_{11}\times R_{2}, \,\,\,\,\ V_{7}= I_{12} \times 0, \,\,\,\,\  V_{8}=I_{12} \times I_{21}, \,\,\,\,\ V_{9}=I_{12}\times I_{22},$$\\ 
 $$V_{10}=I_{12}\times R_{2}, \,\,\,\,\ V_{11}= 0 \times I_{21}, \,\,\,\,\  V_{12}=0 \times I_{22}, \,\,\,\,\ V_{13}=R_{1}\times I_{21}, \,\,\,\,\ V_{14}=R_{1}\times I_{22}$$.
\unitlength=1.5mm
\begin{center}
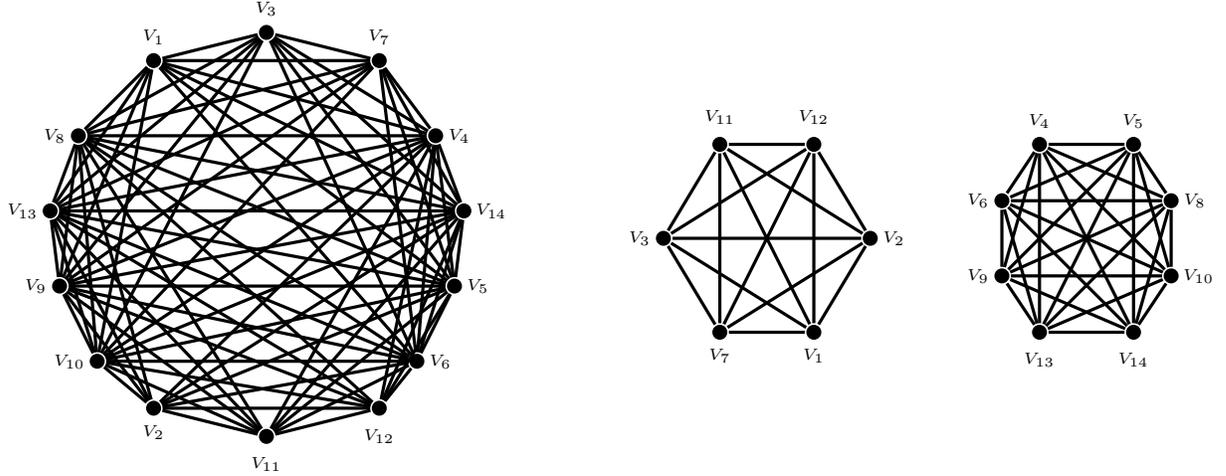
\begin{figure}[H]
 	
\begin{minipage}{.2\textwidth}
\begin{tikzpicture}
  [scale=0.25,every node/.style={circle,fill=black,inner sep=0pt},very thick]
  \node [label=above:{\tiny $V_{1}$},text width=2mm] (n1) at (-4,10) {};
  \node [label=above:{\tiny $V_{3}$},text width=2mm] (n3) at (2,11.5) {};
  \node [label=above:{\tiny $V_{7}$},text width=2mm] (n7) at (8,10) {};
  \node [label=left:{\tiny $V_{8}$},text width=2mm] (n8) at (-8,6) {};
  \node [label=left:{\tiny $V_{9}$},text width=2mm] (n9) at (-9,-2)  {};
  \node [label=left:{\tiny $V_{10}$},text width=2mm] (n10) at (-7,-6)  {};
  \node [label=right:{\tiny $V_{4}$},text width=2mm]  (n4) at (11,6)  {};
  \node [label=right:{\tiny $V_{5}$},text width=2mm] (n5) at (12,-2)  {};
  \node [label=right:{\tiny $V_{6}$},text width=2mm] (n6) at (10,-6)  {};
  \node [label=left:{\tiny $V_{13}$},text width=2mm] (n13) at (-9.5,2) {};
  \node [label=right:{\tiny $V_{14}$},text width=2mm] (n14) at (12.5,2) {};
  \node [label=below:{\tiny $V_{2}$},text width=2mm] (n2) at (-4,-8.5) {};
  \node [label=below:{\tiny $V_{11}$},text width=2mm] (n11) at (2,-10) {};
  \node [label=below:{\tiny $V_{12}$},text width=2mm] (n12) at (8,-8.5) {};
  \foreach \from/\to in {n1/n7,n1/n3,n1/n4,n1/n5,n1/n6,n1/n8,n1/n9,n1/n10,n1/n13,n1/n14, n2/n5,n2/n4,n2/n6,n2/n8,n2/n9,n2/n10,n2/n11,n2/n12,n2/n13,n2/n14,n3/n4,n3/n5,n3/n6,n3/n7,
n3/n8,n3/n9,n3/n10,n3/n13,n3/n14,n4/n5,n4/n6,n4/n7,n4/n8,n4/n9,n4/n10,n4/n11,n4/n12,n4/n13,n4/n14,
n5/n6,n5/n7,n5/n8,n5/n9,n5/n10,n5/n11,n5/n12,n5/n13,n5/n14,
n6/n7,n6/n8,n6/n9,n6/n10,n6/n11,n6/n12,n6/n13,n6/n14,n7/n8,n7/n9,n7/n10,n7/n13,n7/n14,
n8/n9,n8/n10,n8/n11,n8/n12,n8/n13,n8/n14,n9/n10,n9/n11,n9/n12,n9/n13,n9/n14,
n10/n11,n10/n12,n10/n13,n10/n14,n11/n12,n11/n13,n11/n14,n12/n13,n12/n14,n13/n14}
    \draw (\from) -- (\to);
\end{tikzpicture}
\end{minipage}
\hspace{5cm}
\begin{minipage}{.25\textwidth}
\begin{tikzpicture}
  [scale=0.25,every node/.style={circle,fill=black,inner sep=0pt},very thick]
   \node [label=below:{\tiny $V_{1}$},text width=2mm] (n1) at (40,-5)  {};
   \node [label=left:{\tiny $V_{3}$},text width=2mm] (n3) at (32,0)  {};
   \node [label=below:{\tiny $V_{7}$},text width=2mm] (n7) at (35,-5)  {};
   \node [label=right:{\tiny $V_{2}$},text width=2mm] (n2) at (43,0)  {};
  \node [label=above:{\tiny $V_{11}$},text width=2mm] (n11) at (35,5)  {};
  \node [label=above:{\tiny $V_{12}$},text width=2mm] (n12) at (40,5)  {};
  \node [label=above:{\tiny $V_{4}$},text width=2mm]  (n4) at (52,5)  {};
  \node [label=above:{\tiny $V_{5}$},text width=2mm] (n5) at (57,5)  {};
  \node [label=left:{\tiny $V_{6}$},text width=2mm] (n6) at (50,2)  {};
  \node [label=right:{\tiny $V_{8}$},text width=2mm] (n8) at (59,2) {};
  \node [label=left:{\tiny $V_{9}$},text width=2mm] (n9) at (50,-2)  {};
  \node [label=right:{\tiny $V_{10}$},text width=2mm] (n10) at (59,-2)  {};
  \node [label=below:{\tiny $V_{13}$},text width=2mm] (n13) at (52,-5) {};
  \node [label=below:{\tiny $V_{14}$},text width=2mm] (n14) at (57,-5) {};
  
  \foreach \from/\to in {n1/n2,n1/n3,n1/n7,n1/n11,n1/n12,n2/n3,n2/n7,n2/n11,n2/n12,n3/n7,n3/n11,n3/n12,n7/n11,n7/n12,n11/n12,n4/n5,n4/n6,n4/n8,n4/n9,n4/n10,
  n4/n13,n4/n14,n5/n6,n5/n8,n5/n9,n5/n10,n5/n13,n5/n14,n6/n8,n6/n9,n6/n10,n6/n13,n6/n14,n8/n9,n8/n10,n8/n13,n8/n14,n9/n10,n9/n13,n9/n14,
  n10/n13,n10/n14,n13/n14}
    \draw (\from) -- (\to);
\end{tikzpicture}
\end{minipage}

 \caption{ $G(R)$ and $G(R)_{SR}$} \label{figure:fr1}
\end{figure}
\end{center}
}
\end{example}
Then Figure \ref{figure:fr1} shows how  $G(R)_{SR}$ extract from $G(R)$.
\begin{lem}\label{dimprod6}
 
Suppose that  $R\cong \prod_{i=1}^{m}{R}_{i}$, where $R_{i}$ is a PIR  non-field for every $1\leq i\leq m$ and $m\geq 2$ is a positive integer. Then $\beta(G(R)_{SR})=2^{m-1}$.
\end{lem}
\begin{proof}
{
By Lemma \ref{dimprod5}, $G(R)_{SR}= K_{\Pi_{i=1}^{m}(n_{i}+1)-1}+ H$, where $H=G(R)_{SR}[\cup_{i=1}^{m-1}A_{i}]$ is a connected graph. Thus $\beta(G(R)_{SR})=1+ \beta(H)$. We show that $\beta(H)=2^{m-1}-1$. 
Clearly, for every $I, J\in V(H)$ if $d_{G(R)}(I,J)=diam(G(R))$, then $IJ\in G(R)_{SR}$. Therefore, to find the largest independent set in $H$, we have to investigate cliques in $G(R)$.
Let $1 \leq i \leq [\dfrac{m}{2}]-1$, for even $m$ and $1 \leq i \leq [\dfrac{m}{2}]$, for odd $m$ and  $I, J \in A_{i}$. Then $I\cap J \neq 0$ and so $I$ and $J$ are adjacent, i.e.,  $G(R)[A_i]$ is a complete graph. Moreover, if  $I\in A_{i}$ and $J\in A_{j}$ with $1 \leq i\neq j \leq [\dfrac{m}{2}]-1$, for even $m$  and $1 \leq i\neq j \leq [\dfrac{m}{2}]$, for odd $m$, then $I\cap J \neq 0$ and so $I$ and $J$ are adjacent. The above arguments show that $G(R)[A]$ is a complete graph, if one let $A=\cup_{i=1}^{[\frac{m}{2}]}A_{i}$, for odd $m$ and $A=\cup_{i=1}^{[\frac{m}{2}]-1}A_{i}$, for even $m$. Now, let $t=\frac{m}{2}$, where $m$ is even. Then $I$ and $J$ are adjacent in $G(R)$, for every $I\in A_{t}$ and  $J\in A$. We note that if $I\in A_t$, then $I\cap V =0$ and so $IV\in E(G(R)_{SR})$, for every $V \in [I^c]$. This means that the largest independent set $P$ in $A_t$ contains exactly one element from either $[I]$ or  $[I^c]$. Moreover, $|P|= \dfrac{{m\choose t}}{2}$.\\
Now, we are ready to find the largest independent set in $H$. By Lemma \ref{dimprod4},  if $[I]=[J]$, then  $IJ \in E(G(R)_{SR})$, for all $I, J \in V(G(R)_{SR})$. Thus  only one element of the equivalence class $[I]$ can be contained in the largest independent set in $G(R)_{SR}[A]$, for every $I\in A$. On the other hand, the number of equivalence classes in the subgraph induced by every $A_i$ is ${m\choose i}$. Consider the independent set

 $$S= \lbrace I|\, I~is~ representative~ of~ equivalence~ class ~[I],~ for ~every ~I \in A \rbrace,$$

 in $H$. Let $S^{\prime}=S$, for odd $m$ and $S^{\prime}=S\cup P$, for even $m$. Then $S^{\prime}$ is an independent set in $H$. Finally, if $m$ is odd (or even), then  there exists $I\in S^{\prime}$ such that $I$ and $J$ are not adjacent in $G(R)$, for every $J\in V(H)\setminus A$ (or $J\in V(H)\setminus (A\cup A_{t})$). Hence $IJ\in E(G(R)_{SR})$ and so $S^{\prime}\cap (V(H)\setminus A)=\emptyset$ (or $S^{\prime}\cap V(H)\setminus (A\cup A_{t})=\emptyset$). Furthermore, $|S^{\prime}|={m \choose 1} +\cdots+ {m \choose t}= 2^{m-1}-1$, where $m$ is odd and $|S^{\prime}|={m \choose 1} +\cdots+ {m \choose t-1}+ \dfrac{{m \choose t}}{2}= 2^{m-1}-1$, where $m$ is even. Thus  $S^{\prime}$ is the largest independent subset of $V(H)$ of order $2^{m-1}-1$ and so $\beta(H)=|S^{\prime}|= 2^{m-1}-1$. 
}
\end{proof}

\begin{thm}\label{isomorphism3}
 
Suppose that  $R\cong \prod_{i=1}^{m}{R}_{i}$, where $R_{i}$ is a PIR  non-field for every $1\leq i\leq m$ and $m\geq 2$ is a positive integer. Then $sdim(G(R))=\Pi_{i=1}^{m}(n_{i}+2)-2^{m-1}-2$.
\end{thm}
\begin{proof}
{ By Lemma \ref{dimprod6}, $\beta(G(R)_{SR})=2^{m-1}$.  Since
$|V(In(R)_{SR})|=\Pi_{i=1}^{m}(n_{i}+2)-2$, Gallai$^{^,}$s theorem and Lemma \ref{Oellermann} show that $sdim(G(R))=|V(G(R)_{SR})|-\beta(G(R)_{SR}) =\Pi_{i=1}^{m}(n_{i}+2)-2^{m-1}-2$.
 }
\end{proof}

Finally, we investigate $sdim(G(R))$, where both of fields and non-fields appear in decomposition of $R$.

\begin{lem}\label{dimprod7}
Let  $R\cong S\times T$ such that $S= \prod_{i=1}^{m}{R}_{i}$, $T=\prod_{j=1}^{n}\mathbb{F}_{j}$, where $R_{i}$ is a PIR  non-field for every $1\leq i\leq m$, $\mathbb{F}_{j}$ is a field  for every $1\leq j\leq n$ and $m,n\geq 1$ are positive integers. Then the following statements hold:\\
$1)$ $V(G(R)_{SR})=V(G(R))$.\\
$2)$ For every $I, J \in V(G(R)_{SR})$, if $[I]=[J]$, then  $IJ \in E(G(R)_{SR})$.\\
$3)$ For every $I, J\in V(G(R)_{SR})$, if $[I]\neq [J]$, then $IJ \in E(G(R)_{SR})$ if and only if $IJ \notin E(G(R))$.
\end{lem}
\begin{proof}{
It is enough to apply a similar argument to that of Lemma \ref{dimprod4}.
}
\end{proof}
\begin{lem}\label{dimprod8}
Let  $R\cong S\times T$ such that $S= \prod_{i=1}^{m}{R}_{i}$, $T=\prod_{j=1}^{n}\mathbb{F}_{j}$, where $R_{i}$ is a PIR  non-field for every $1\leq i\leq m$, $\mathbb{F}_{j}$ is a field  for every $1\leq j\leq n$ and $m,n\geq 1$ are positive integers. Then $G(R)_{SR}= K_{\Pi_{i=1}^{m}(n_{i}+1)2^{n}}+ H $, where $H$ is a connected graph.
\end{lem}
\begin{proof}{
 By Lemma \ref{dimprod7}, $V(G(R))=V(G(R)_{SR})$. Also, since for every $I, J\in A_{0}$, $[I]=[J]$ so $IJ\in E(G(R)_{SR})$. Thus induced subgraph $G(R)[A_{0}]$ is a complete graph. Also, by Lemma \ref{dimprod7}, for every $I\in A_{0}$ and for every $J\in V(G(R)_{SR})\setminus A_{0}$, $IJ\notin E(G(R)_{SR})$. Furthermore, $|A_{0}|= \Pi_{i=1}^{m}(n_{i}+1)2^{n}$. Thus $G(R)[A_{0}]= K_{\Pi_{i=1}^{m}(n_{i}+1)}2^{n}$.\\ To complete the proof, it is enough to apply a similar argument to that of Lemma \ref{dimprod5}.
}
\end{proof}
\begin{lem}\label{dimprod9}
 Let  $R\cong S\times T$ such that $S= \prod_{i=1}^{m}{R}_{i}$, $T=\prod_{j=1}^{n}\mathbb{F}_{j}$, where $R_{i}$ is a PIR  non-field for every $1\leq i\leq m$, $\mathbb{F}_{j}$ is a field  for every $1\leq j\leq n$ and $m,n\geq 1$ are positive integers. Then $\beta(G(R)_{SR})=2^{m+n-1}$.
\end{lem}
\begin{proof}
{
 By Lemma \ref{dimprod8}, $G(R)_{SR}=  K_{\Pi_{i=1}^{m}(n_{i}+1)2^{n}} + H$, so $\beta(G(R)_{SR})= \beta(H)+ 1$. Also, by a similar argument to that of Lemma \ref{dimprod6}  and  case $(1)$ of Lemma \ref{dimprod7}, $S= \cup_{i=1}^{[\frac{m+n}{2}]}A_{i}$, where $m+n$ is odd and $S= \cup_{i=1}^{[\frac{m+n}{2}]-1}A_{i}$ union with half of the members of $A_{\frac{m+n}{2}}$ , where $m+n$ is even  is the largest independent subset of $V(H)$ and $|S|= 2^{m+n-1}-1 $. Hence $\beta(G(R)_{SR})=|S| + 1= 2^{m+n-1}$.
}
\end{proof}

We close this paper with the following result.

\begin{thm}\label{isomorphism4}
 Let  $R\cong S\times T$ such that $S= \prod_{i=1}^{m}{R}_{i}$, $T=\prod_{j=1}^{n}\mathbb{F}_{j}$, where $R_{i}$ is a PIR  non-field for every $1\leq i\leq m$, $\mathbb{F}_{j}$ is a field  for every $1\leq j\leq n$ and $m,n\geq 1$ are positive integers. Then $sdim(G(R))=\Pi_{i=1}^{m}(n_{i}+2)2^{n}- 2^{m+n-1}-2$.
\end{thm}
\begin{proof}
{
By Lemma \ref{dimprod9}, $\beta(G(R)_{SR})=2^{m+n-1}$. Since
$|V(G(R)_{SR})|=\Pi_{i=1}^{m}(n_{i}+1)2^{n}-2$, Gallai$^{^,}$s theorem and Lemma \ref{Oellermann} show that $sdim(G(R))=|V(G(R)_{SR})|-\beta(G(R)_{SR}) =\Pi_{i=1}^{m}(n_{i}+1)2^{n}- 2^{m+n-1}-2$.
}
\end{proof}


{}


\begin{thebibliography}{}{\small




\bibitem{abr} G. Abrishami, M. A. Henning, M. Tavakoli, Local metric dimension for graphs with small clique numbers, Discrete Math. 345 (2022) 112763.

\bibitem{bakht} N. Abachi,  M. Adlifard,  M. Bakhtyiari, Strong metric dimension of a total graph of nonzero annihilating ideals, Bull.   Aust. Math. Soc.  105 (2022) 431--439.

\bibitem{akb} S. Akbari, R. Nikandish, M. J. Nikmehr, Some results on the intersection graphs of ideals of rings, J. Algebra its Appl. 12 (2013) (04), 1250200.

\bibitem{ali} F. Ali, M. Salman, S. Huang, On the commuting graph of dihedral group, Comm. Algebra  44 (2016) 2389--2401.


\bibitem{ati} M. F. Atiyah, I. G. Macdonald,  Introduction to
Commutative Algebra, Addison-Wesley Publishing Company (1969).

\bibitem{bai} R. F. Bailey, P. Spiga, Metric dimension of dual polar graphs, Arch Math 120 (2023) 467-–478.


\bibitem{Chak}
I. Chakrabarty, S. Ghosh, T. K. Mukherjee, M. K. Sen, Intersection graphs of ideals of rings, Discrete Math. 309 (2009) 5381--5392.



\bibitem{dol} D. Dol\v{z}an, The metric dimension of the total graph of a finite commutative ring, Canad. Math. Bull.  59 (2016) 748--759.

\bibitem{dol2} D. Dol\v{z}an, The metric dimension of the annihilating-ideal graph of a finite commutative ring, Bull.   Aust. Math. Soc. 103  (2021)  362--368.

\bibitem{ghala} A. Ghalavand, S. Klavža,  M. Tavakoli, Graphs whose mixed metric dimension is equal to their order, 42 (2023) Article number: 210. 


 \bibitem{Ji} Z. Jiang, N. Polyanskii, On the metric dimension of cartesian powers of a graph, J. Comb.
theory Ser. A  165 (2019) 1--14.
\bibitem{ma} X. Ma, L. Li, On the metric dimension of the reduced power graph of a finite group, Taiwanese J. Math. 26 (2022) 1-15.

\bibitem{mah} A. Mahmoodi, A. Vahidi, R. Manaviyat, R. Alipour, Intersection graph of idealizations, Collectanea Mathematica (2023). https://doi.org/10.1007/s13348-023-00407-7



\bibitem{nik} R. Nikandish,  Investigating the metric dimension of an intersection  graph in a commutative ring, Karafan 17 (2021) 35--44.

\bibitem{nili} R. Nikandish,  M. J. Nikmehr,  M. Bakhtyiari, Strong resolving graph of a zero-divisor graph, Revista de la Real Academia de Ciencias Exactas, Físicas y Naturales. Serie A. Matemáticas 116 (2022) Article number: 116.

\bibitem{nili2} R. Nikandish,  M. J. Nikmehr,  M. Bakhtyiari, Metric and Strong Metric Dimension in Cozero-Divisor Graphs, Mediterr. J. Math.  18 (2021) Article number: 112.

\bibitem{oller} O. R. Oellermann, J. Peters-Fransen, The strong metric dimension of graphs and digraphs, Discrete Appl. Math. 155 (2007) 356--364.


\bibitem{Pirzada1} S. Pirzada, R. Raja, On the metric dimension of a zero-divisor graph,
Comm. Algebra 45 (2017) 1399--1408.




 \bibitem{Pirzada2} R. Raja, S. Pirzada  and S. P. Redmond, On Locating numbers and codes of zero-divisor graphs associated with commutative rings,
 J. Algebra Appl. 15 (2016) 1650014 (22 pages).


\bibitem{seb} A. Seb\"{o}, E. Tannier, On metric generators of graphs, Math. Oper. Res. 29 (2004) 383--393.

\bibitem{Ve} T. Vetr\'{\i}k, The metric dimension of circulant graphs, Canad.
Math. Bull. 60 (2017)  206--216.


\bibitem{west} D. B. West, Introduction to Graph Theory, 2nd ed., Prentice Hall, Upper Saddle River (2001).

\bibitem{wu} J. Wu, L, Wang, W. Yang, Learning to compute the metric dimension of graphs, Appl. Mat. Comput. 432 (2022) 127350.

\bibitem{xu} F. Xu, D. Wong, F. Tian, Automorphism group of the intersection graph of ideals over a matrix ring, Linear and Multilinear Algebra 70 (2022) 322--330.

\bibitem{zai} L. Zhai, X. Ma, Y. Shao, G. Zhong, Metric and strong metric dimension in commuting graphs of finite groups, Comm. Algebra. 51 (2023) 1000--1010.

}
\end{thebibliography}
\end{document}